\documentclass[10pt, reqno]{amsart} 

\usepackage{amsmath, amsthm, amssymb, amsfonts}
\usepackage{mathrsfs}
\usepackage{enumerate}
\usepackage{enumitem}
\usepackage{graphicx}
\usepackage{caption, subcaption}
\usepackage{float}
\usepackage{mathtools}

\usepackage[colorlinks=true, linkcolor=blue, citecolor=red, urlcolor=blue]{hyperref}

\usepackage{lineno} 

\usepackage[colorinlistoftodos, textsize=footnotesize]{todonotes}

\theoremstyle{plain}
\newtheorem{theorem}{Theorem}[section]
\newtheorem{corollary}[theorem]{Corollary}
\newtheorem{lemma}[theorem]{Lemma}

\theoremstyle{definition}

\theoremstyle{remark}

\newcommand{\R}{\mathbb{R}}
\newcommand{\e}{\varepsilon}

\begin{document}
\allowdisplaybreaks

\title[A Bounded Determinantal Ratio for Positive Definite Matrices]{A Bounded Determinantal Ratio for Positive Definite Matrices}

\author{Hristo Sendov}
\address{Department of Statistical and Actuarial Sciences \\
Department of Mathematics \\
Western University \\
1151 Richmond Str. \\
London, ON, N6A 5B7 Canada}
\email{hsendov@uwo.ca}
\thanks{The author was partially supported by the Natural Sciences and Engineering Research Council (NSERC) of Canada. (Grant number RGPIN-2020-06425.)}


\date{}


\begin{abstract}
Bounded ratios between products of minors of a positive definite matrix have a long history. Starting with Hadamard's inequality which bounds from above the determinant by the product of its diagonal entries and progressing through its generalizations the Fischer’s inequality and the Koteljanskii’s inequality. Finding new bounded determinantal ratios is a difficult task and finding their supremum is even more difficult. In 2008, Hall and Johnson showed that the ratio
$$
\frac{\det A[\{1,2,4\}] \det A[\{1,3,4\}] \det A[\{2,3\}] \det A[\{1\}] \det A[\{4\}]}{\det A[\{1,2\}] \det A[ \{1,3\}] 
\det A[\{1,4\}]  \det A[\{2,4\}] \det A[\{3,4\}]} 
$$
is bounded above by $4$. They conjectured that the supremum of the ratio, over all $4 \times 4$ positive definite matrices $A$, is $27/16$. Here, 
$A[\alpha]$ denotes the principal minor of $A$ corresponding to the rows and columns indexed by $\alpha \subseteq \{1,2,3,4\}$. In this paper we confirm that the supremum of the ratio is $27/16$ and exhibit a sequence of matrices that approaches it. 
\end{abstract}

\maketitle

\section{Introduction}

Let $A := (a_{ij})$ denote an $n$-by-$n$ positive definite matrix. Let $\alpha := \{\alpha_1, \dots, \alpha_p\}$ be a collection of index sets, such that each $\alpha_i \subseteq \{1,\dots,n\}$. For an index set $\alpha_i$, define $A[\alpha_i]$ to be the principal submatrix of $A$ corresponding to the rows and columns indexed by $\alpha_i$. By convention, if $\alpha_i$ is empty, then $A[\emptyset] := 1$.  Define the product of determinants over these index sets as:
$$
A(\alpha) := \prod_{i=1}^p \det A[\alpha_i].
$$
Much effort has been made to characterize the collection of sets $\alpha$ and $\beta$ for which the 
ratio   $A(\alpha)/A(\beta)$ is bounded (more specifically, bounded by one) for all positive definite matrices $A$. An example of these types of ratios comes from Koteljanskii’s inequality, also known as the Hadamard-Fischer inequality:
\begin{align}
\label{2022-07-27-ineq-3}
\frac{\det A[\alpha_1 \cup \alpha_2] \det A[\alpha_1 \cap \alpha_2] }{\det A[\alpha_1] \det A[\alpha_2]} \leq 1,
\end{align}
which holds for any $\alpha_1, \alpha_2 \subseteq \{1,\dots,n\}$ and any positive definite $A$. 
This is an example of a  ratio bounded by one.
Koteljanskii’s inequality generalizes both Hadamard’s inequality:
\begin{align*}
\frac{\det A}{a_{11} \cdots a_{nn}} \leq 1,
\end{align*}
and Fischer’s inequality:
\begin{align*}
\frac{\det A}{\det A[\alpha] \det A[\alpha^c]} \leq 1,
\end{align*}
where $\alpha^c$ is the complement of $\alpha$ in $\{1,\dots,n\}$. These inequalities are detailed in \cite{Horn:1990}, see Theorems 7.8.1, 7.8.5, and 7.8.9. Various generalizations and extensions of inequality~\eqref{2022-07-27-ineq-3} appear in \cite{Barrett:1989}, \cite{Braun:2023},  \cite{Choi:2016}, \cite{Dong:2022}, \cite{Fallat:2001}, \cite{Fallat:2003}, \cite{Fu:2017},\cite{Johnson:1985},  \cite{Johnson:1993}   and \cite{Jiang:2019} among others. 

Finding other bounded determinantal ratios turns out to be a very difficult problem. Necessary and sufficient conditions (but not both) were formulated in \cite{Johnson:1993}. In \cite{TracyHall:2008}, the authors showed, using results from \cite{Fiedler:1964}, that
\begin{equation}
R(A) := \frac{A(\{\{1,2,4\}, \{1,3,4\}, \{2,3\}, \{1\}, \{4\} \})}{A(\{ \{1,2\}, \{1,3\}, \{1,4\}, \{2,4\}, \{3,4\} \})} \le 4.
\end{equation}
They conjectured that the smallest upper bound is $27/16$.  

In this paper, we confirm that $27/16$ is the smallest upper bound on $R(A)$ and exhibit a sequence of positive definite matrices that approximates it. The result makes this the first example of a bounded determinantal ratio with supremum bigger than one.  The expanded ratio is 
$$
R(A) = \frac{\left| \begin{array}{lll} a_{11} & a_{12} & a_{14} \\ a_{12} & a_{22} & a_{24} \\ a_{14} & a_{24} & a_{44} \end{array} \right|
\left| \begin{array}{lll} a_{11} & a_{13} & a_{14} \\ a_{13} & a_{33} & a_{34} \\ a_{14} & a_{34} & a_{44} \end{array} \right|
\left| \begin{array}{ll} a_{22} & a_{23} \\ a_{23} & a_{33} \end{array} \right| a_{11} a_{44} }{
 \left| \begin{array}{ll} a_{11} & a_{12} \\ a_{12} & a_{22} \end{array} \right|
  \left| \begin{array}{ll} a_{11} & a_{13} \\ a_{13} & a_{33} \end{array} \right|
   \left| \begin{array}{ll} a_{11} & a_{14} \\ a_{14} & a_{44} \end{array} \right|
    \left| \begin{array}{ll} a_{22} & a_{24} \\ a_{24} & a_{44}\end{array} \right|
     \left| \begin{array}{ll} a_{33} & a_{34} \\ a_{34} & a_{44} \end{array} \right|
     }.
$$
By dividing appropriate rows and columns in the numerator and the denominator by the square root of the positive entries $a_{11}, a_{22}, a_{33}, a_{44}$, one can assume that $A$ is a correlation matrix, that is
$$
a_{11} =  a_{22} =  a_{33} = a_{44} = 1.
$$
The positive definiteness implies that  $a_{ij} \in (-1,1)$ for all $1\le i < j \le 4$. 
By multiplying rows and corresponding columns in the numerator and the denominator by $-1$, if necessary, we can further assume that 
\begin{align*}
a_{12},  a_{13},  a_{14}, a_{23} \in [0, 1) \mbox{ and }  a_{24},  a_{34} \in (-1,1).
\end{align*}
These assumptions are used only to shorten a few case studies, otherwise they are not essential for the development. 
Final important observation is that this ratio is invariant under swapping of the indexes $1 \leftrightarrow 4$ and $2 \leftrightarrow 3$. The following theorem is the main result of the paper. 

\begin{theorem}
\label{2026-01-15-mainthm}
For any $4\times 4$ positive definite matrix $A$ we have $R(A) < 27/16$ and that supremum is approximated by 
\begin{align*}
A(\epsilon):=
\begin{bmatrix}
1 & 1-\epsilon & 1-\epsilon  & 1-\epsilon  \\
1-\epsilon  & 1 & 1-3\epsilon+\epsilon^2 & 1-\epsilon \\
1-\epsilon  & 1-3\epsilon+\epsilon^2   & 1 & 1-\epsilon  \\
1-\epsilon & 1-\epsilon  &1-\epsilon  & 1
\end{bmatrix}
\end{align*}
as $\epsilon$ approaches $0^+$. 
\end{theorem}

Our approach to showing that $R(A)$ is bounded above by $27/16$ is a brute force. We look at the problem as an optimization problem and break the solution into steps discussed in the following subsections.  If one looks at the details of the proof of Theorem~\ref{2026-01-15-mainthm}, it is unavoidable to see that the six variables work together in a magical unison and the supremum $27/16$ is very delicate. This hints that there may be a bigger picture behind the mysterious bounded determinantal ratios. We hope that our investigation sheds light in that direction. 

The domain over which we are optimizing is 
$$
D:=\{A \in \R^{4 \times 4}: A \succ 0, \mbox{diag\,}(A) = (1,1,1,1)\}.
$$
In Subsection~\ref{subsec2.1} we show that the supremum of $R$ is not attained at a (relative) interior point of the domain.
Thus the supremum is attained by a sequence of interior points converging to a boundary point. The boundary consists of positive semidefinite correlation matrices $A$ with zero determinant. They fall into two categories: 1) those that make the denominator of $R$ zero and  2) those that do not. Boundary matrices of the first type are investigated in Subsection~\ref{subsec2.2}, while those of the second type are investigated in Subsection~\ref{subsec2.3}.

\section{Proof of Theorem~\ref{2026-01-15-mainthm}}

The facts that $A(\epsilon)$ is positive definite and that $R(A(\epsilon))$ approaches $27/16$ as $\epsilon \to 0^+$
can be verified directly. Thus, we know that $\sup R \in [27/16, 4]$. Let
\begin{align*}
\Delta_{ij} (A):= \det A(\{i,j\}) = 1-a_{ij}^2, \quad \mbox{for $1 \le i < j \le 4$}.
\end{align*}
Also for $1 \le i < j < k \le 4$, let
\begin{align}
\label{2026-01-15-Delta}
\Delta_{ijk}(A) := \det A(\{i,j,k\}) &= 1+2a_{ij}a_{ik}a_{jk}-a_{ij}^2-a_{ik}^2-a_{jk}^2 \\
\nonumber
&= (1-a^2)(1-b^2)-(c-ab)^2 
\end{align}
for all $a,b,c \in \{a_{ij}, a_{jk}, a_{ik}\}$. Thus, if $a \rightarrow 1$, then the condition $\Delta_{ijk} > 0$ implies $b-c \rightarrow 0$. We use that implication frequently.  

\subsection{\texorpdfstring{$R$}{R} does not achieve its maximum in the interior of the domain.}
\label{subsec2.1}
The ratio can be written as
\begin{align*}
R=\frac{\Delta_{124}\,\Delta_{134} \, \Delta_{23}}
{\Delta_{12} \, \Delta_{13} \, \Delta_{14} \, \Delta_{24} \, \Delta_{34} }.
\end{align*}
Taking logarithms from both sides gives
\begin{equation}\label{eq:Phi}
\Phi(A):= 
\log\Delta_{124}+\log\Delta_{134}+\log(\Delta_{23})
-\sum_{ij\in\{12,13,14,24,34\}}\log(\Delta_{ij} ).
\end{equation}
If $R(A)$ achieves its maximum in an interior point then the partial derivatives need to be zero there. 
Note that only the third term depends on $a_{23}$. Differentiating \eqref{eq:Phi} with respect to $a_{23}$ gives
\begin{align*}
\frac{\partial\Phi}{\partial a_{23}}
=-\frac{2a_{23}}{1-a_{23}^2} = 0 \mbox{ or } a_{23}=0.
\end{align*}
From $\partial\Phi/\partial a_{12}=0$ and $\partial\Phi/\partial a_{24}=0$ one gets
\begin{align}
\label{E12}
\frac{a_{12}-a_{14}a_{24}}{\Delta_{124}}=\frac{a_{12}}{1-a_{12}^2}
\quad  \mbox{and} \quad
\frac{a_{24}-a_{12}a_{14}}{\Delta_{124}}=\frac{a_{24}}{1-a_{24}^2}.
\end{align}
Eliminating $\Delta_{124}$ from these equations shows that 
\begin{align}
\label{2026-01-15-cond1}
\mbox{if $a_{12}\neq 0$ and $a_{24}\neq 0$, then } (a_{12}^2-a_{24}^2)(a_{12}a_{24}-a_{14})=0.
\end{align}
Suppose that $a_{12}\neq 0$, $a_{24}\neq 0$, and $a_{12} = \sigma a_{24}$, where $\sigma\in\{+1,-1\}$.
Substitute into the second equation in \eqref{E12} and after cancelling $a_{24}$ we obtain
$$
\Delta_{124}=(1-a_{24}^2)(1-\sigma a_{14}) = 1+\sigma a_{14} a_{24}^2 - \sigma a_{14}- a_{24}^2.
$$
Equating with the expression for $\Delta_{124}$ coming from \eqref{2026-01-15-Delta} and factoring, results in
$$
(\sigma a_{14} - 1)(a_{14} - \sigma a_{24}^2) = 0.
$$
Since we are at an interior point, $a_{14} \not= \pm 1$, thus $a_{14} = \sigma a_{24}^2 = a_{12}a_{24}$. 
Comparing with \eqref{2026-01-15-cond1} gives the condition
 \begin{align}
 \label{2026-01-15-cond2a}
\mbox{either $a_{12} = 0$, or $a_{24} = 0$, or else } a_{14}= a_{12}a_{24}.
\end{align}

Next, $\partial\Phi/\partial a_{13}=0$ and $\partial\Phi/\partial a_{34}=0$ give
\begin{align}
\frac{a_{13}-a_{14}a_{34}}{\Delta_{134}}=\frac{a_{13}}{1-a_{13}^2}
\label{E13}
\quad  \mbox{and} \quad
\frac{a_{34}-a_{13}a_{14}}{\Delta_{134}}=\frac{a_{34}}{1-a_{34}^2}.
\end{align}
Eliminating $\Delta_{134}$ from these equations shows that 
\begin{align}
\label{2026-01-15-cond2}
\mbox{if $a_{13}\neq 0$ and $a_{34}\neq 0$, then } (a_{13}^2-a_{34}^2)(a_{13}a_{34}-a_{14})=0
\end{align}
and proceeding analogously to above, we conclude that 
 \begin{align}
\label{2026-01-15-cond3}
\text{either } a_{13}=0 \text{ or } a_{34}=0,\ \text{or else } a_{14}=a_{13}a_{34}.
\end{align}
The last partial derivative $\partial\Phi/\partial a_{14}=0$, gives
\begin{align}
\label{2026-01-15-cond4}
\frac{a_{14}-a_{12}a_{24}}{\Delta_{124}}
+
\frac{a_{14}-a_{13}a_{34}}{\Delta_{134}}
=
\frac{a_{14}}{1-a_{14}^2}.
\end{align}
We compare \eqref{2026-01-15-cond2a} and \eqref{2026-01-15-cond3} with \eqref{2026-01-15-cond4} in the next four short cases.

\smallskip

\noindent\text{Case 1:} $a_{14}=a_{12}a_{24}$ and $a_{14}=a_{13}a_{34}$.
Then the left-hand side of \eqref{2026-01-15-cond4} is $0$, forcing $a_{14}=0$ and thus 
$a_{12}a_{24}= a_{13}a_{34}=0$. 

\smallskip

\noindent\text{Case 2:} $a_{14}=a_{12}a_{24}$ and $a_{14}\neq a_{13}a_{34}$.
Then \eqref{2026-01-15-cond3} implies $a_{14} \not= a_{13} a_{34} = 0$. Plugging everything into 
\eqref{2026-01-15-cond4} implies that  $a_{13}=a_{34}=0$.

\smallskip

\noindent\text{Case 3:} $a_{14}\neq a_{12}a_{24}$ and $a_{14}=a_{13}a_{34}$.
Then \eqref{2026-01-15-cond2a} implies $a_{14} \not= a_{12} a_{24} = 0$. Plugging everything into 
\eqref{2026-01-15-cond4} implies that  $a_{12}=a_{24}=0$.

\smallskip

\noindent\text{Case 4:} $a_{14}\neq a_{12}a_{24}$ and $a_{14} \not=a_{13}a_{34}$.
Then \eqref{2026-01-15-cond2a} and \eqref{2026-01-15-cond3} imply $a_{12}a_{24}=0$ and $a_{13}a_{34}=0$ and $a_{14} \not=0$. If $a_{12} \not= 0$, then  $a_{24}=0$ and substituting in \eqref{E12} we get $a_{14}^2+a_{24}^2=0$, which is a contradiction, hence $a_{12}= a_{24}=0$. Similarly \eqref{E13} implies that $a_{13}=a_{34}=0$.

A straightforward calculation shows that if a correlation matrix satisfies the conditions and the conclusions of each of these cases, keeping in mind that $a_{23}=0$, then $R(A)=1$. Since the supremum of $R$ is bigger than $1$, an interior point cannot attain the supremum.

\subsection{Boundary matrices that make the denominator of \texorpdfstring{$R$}{R} zero}
\label{subsec2.2}

In this subsection, we take a sequence of positive definite correlation matrices that approaches a boundary matrix that makes the denominator of $R$ zero. We begin by clearing several preliminary cases that show that the only interesting case is when the limit of the sequence is the all-one matrix.  Since the diagonal entries of the correlation matrices are one, by Koteljanskii's inequality, we have
\begin{align*}
\Delta_{124} &\le \min \{\Delta_{12}\Delta_{14}, \Delta_{12} \Delta_{24}, \ \Delta_{14} \Delta_{24}\}, \\
\Delta_{134} &\le \min \{\Delta_{13}\Delta_{14}, \Delta_{13} \Delta_{34}, \Delta_{14} \Delta_{34}\}.
\end{align*}
We consider several brief cases. 

\smallskip

\noindent
{Case 1}. Suppose that $a_{23} \rightarrow 1$, then using the Koteljanskii's inequalities 
$$
R=\frac{\Delta_{124}\,\Delta_{134} \, \Delta_{23}}
{\Delta_{12} \, \Delta_{13} \, \Delta_{14} \, \Delta_{24} \, \Delta_{34} } \le \frac{\Delta_{12}\, \Delta_{14}\,\Delta_{13} \, \Delta_{34}\, \Delta_{23}}
{\Delta_{12} \, \Delta_{13} \, \Delta_{14} \, \Delta_{24} \, \Delta_{34} } = \frac{\Delta_{23}}{\Delta_{24}}.
$$
Thus $R$ converges to $0$ unless $\Delta_{24} \rightarrow 0$ as well. Similarly, mixing the different Koteljanskii's 
inequalities, so that exactly one $\Delta_{14}$ appears in the numerator, we obtain 
\begin{align}
\label{2026-01-23-Rmin}
R\le \min \Big\{\frac{\Delta_{23}}{\Delta_{12}}, \frac{\Delta_{23}}{\Delta_{13}}, \frac{\Delta_{23}}{\Delta_{24}}, \frac{\Delta_{23}}{\Delta_{34}} \Big\}.
\end{align}
This implies that for something interesting to happen, we need to also have  
$a_{12}$, $a_{13}$, $|a_{24}|$, $|a_{34}|$  converge to $1$. The conditions $\Delta_{124} > 0$ and $\Delta_{134} > 0$ now imply that $a_{14}-a_{24} \to 0$ and $a_{14}-a_{34} \to 0$ and since $a_{14} \ge 0$, we conclude that $a_{14}$, $a_{24}$, and $a_{34}$ converge to $1$. Thus all variables converge to $1$. We deal with this case below.

\smallskip

\noindent
{Case 2}. Suppose that $a_{23} \not\rightarrow 1$ but $a_{14} \rightarrow 1$. 
Then we apply the Koteljanskii's inequalities mixing them, so that two $\Delta_{14}$ appear in the numerator
$$
R=\frac{\Delta_{124}\,\Delta_{134} \, \Delta_{23}}
{\Delta_{12} \, \Delta_{13} \, \Delta_{14} \, \Delta_{24} \, \Delta_{34} } 
\le \frac{\Delta_{12}\, \Delta_{14}\,\Delta_{13} \, \Delta_{14}\, \Delta_{23}}
{\Delta_{12} \, \Delta_{13} \, \Delta_{14} \, \Delta_{24} \, \Delta_{34} } = \frac{\Delta_{14}\, \Delta_{23}}{\Delta_{24} \, \Delta_{34} }.
$$
The ratio will converge to zero, unless the first implication below holds, with the rest being similar  
\begin{align*}
&\mbox{either } |a_{24}| \rightarrow 1 \mbox{ or } |a_{34}| \rightarrow 1, \\
&\mbox{either } \phantom{|}a_{12}\phantom{|} \rightarrow 1 \mbox{ or } |a_{34}| \rightarrow 1,\\
&\mbox{either } \phantom{|}a_{12}\phantom{|} \rightarrow 1 \mbox{ or } \phantom{|}a_{13}\phantom{|} \rightarrow 1,\\
&\mbox{either } |a_{24}| \rightarrow 1 \mbox{ or } \phantom{|}a_{13}\phantom{|} \rightarrow 1.
\end{align*}
If $a_{12} \rightarrow 1$, then $\Delta_{123} > 0$ implies $a_{23}-a_{13} \rightarrow 0$.
Hence $\Delta_{23}/\Delta_{13} \rightarrow 1$ and \eqref{2026-01-23-Rmin} shows that $R \le 1$.
If $|a_{34}| \to 1$ then $\Delta_{134} > 0$ implies $a_{13}-a_{34} \rightarrow 0$, but since $a_{13} \ge 0$, this can only happen if $a_{34} \to 1$ in which case $a_{13} \to 1$. Then $\Delta_{123} >0$ implies that $a_{12} - a_{23} \to 0$. 
Hence $\Delta_{23}/\Delta_{12} \rightarrow 1$ and \eqref{2026-01-23-Rmin} shows that $R \le 1$. The other cases are similar. 

\smallskip

\noindent
{Case 3}. Suppose that $a_{23} \not\rightarrow 1$, $a_{14} \not\rightarrow 1$, but $|a_{ij}| \rightarrow 1$ for some $ij \in \{12, 13,24,34\}$.  Suppose $|a_{24}| \to 1$, then $\Delta_{234} >0$ implies that $a_{23} - |a_{34}| \to 0$ and we conclude as before. The other cases are analogous.

\smallskip 

The above discussion shows that the all-one matrix is the only interesting boundary matrix that makes the denominator of $R$ zero. Take a sequence $\{A_n\} = \{(a_{ij}^n)\}$
 of $4 \times 4$ positive definite correlation matrices approaching the all-one matrix.  Define the gaps to be
\[
\delta_{ij}^{n}:=1-a_{ij}^{n}>0 \mbox{ for all } ij = 12,13,14,23, 24,34
\]
let $\varepsilon_n:=\max\{\delta_{ij}^{n} : ij = 12,13,14,24,34\}$ and note that $\varepsilon_n\to 0$.
Define
\[
x_{ij}^n:=\frac{\delta_{ij}^{n}}{\varepsilon_n} \in (0,1] \mbox{  for all $ij = 12,13,14,24,34$ and } x_{23}^n:=\frac{\delta_{23}^{n}}{\varepsilon_n} \ge 0.
\]
The sequence $\{x_{23}^n\}$ is also bounded, as the next lemma shows.

\begin{lemma}
We have $x_{23}^n \in [0,4]$ for all $n$.
\end{lemma}

\begin{proof}
Begin by bounding 
\begin{align*}
 1-a_{12}^{n}a_{13}^{n}
&=\delta_{12}^{n}+\delta_{13}^{n}-\delta_{12}^{n}\delta_{13}^{n}
\le \delta_{12}^{n}+\delta_{13}^{n}  \mbox{ and }\\
1-(a_{12}^{n})^2
&=1-(1-\delta_{12}^{n})^2
=2\delta_{12}^{n}-(\delta_{12}^{n})^2
\le 2\delta_{12}^{n},
\end{align*}
and similarly $1-(a_{13}^{n})^2\le 2\delta_{13}^{n}$. 

Expand the determinant  $\Delta_{123}^{n}$ 
\[
\Delta_{123}^{n}
= \bigl(1-(a_{12}^{n})^2\bigr)\bigl(1-(a_{13}^{n})^2\bigr) 
- \bigl(a_{23}^{n}-a_{12}^{n}a_{13}^{n}\bigr)^2
>0.
\]
and solve for $a_{23}^{n}$ to get the interval constraint
\[
a_{12}^{n}a_{13}^{n}-\sqrt{(1-(a_{12}^{n})^2)(1-(a_{13}^{n})^2)}
<
a_{23}^{n}
<
a_{12}^{n}a_{13}^{n}+\sqrt{(1-(a_{12}^{n})^2)(1-(a_{13}^{n})^2)}.
\]
From the lower interval constraint we obtain
\begin{align*}
1-a_{23}^{n}
&< 1-a_{12}^{n}a_{13}^{n} +\sqrt{(1-(a_{12}^{n})^2)(1-(a_{13}^{n})^2)} \\
&\le \delta_{12}^{n}+\delta_{13}^{n} + 2\sqrt{\delta_{12}^{n}\delta_{13}^{n}} \\
& \le 2\bigl(\delta_{12}^{n}+\delta_{13}^{n}\bigr) \\
& \le 4\varepsilon_n.
\end{align*}
Thus, $x_{23}^n= (1-a_{23}^{n})/\varepsilon_n \le 4$.
\end{proof}

Vector $(x_{12}^n, x_{13}^n, x_{14}^n, x_{24}^n, x_{34}^n, x_{23}^n)$ is bounded, hence has a convergent subsequence.
Thus, without loss of generality, we may assume that
\[
(x_{12}^n, x_{13}^n, x_{14}^n, x_{24}^n, x_{34}^n, x_{23}^n) \to (x,y,z,u,v,t) \in[0,1]^4 \times [0,4]
\]
and that the sequence $\{A_n\}$ has the form
\begin{align*}
A_n =
\begin{pmatrix}
1& 1-\varepsilon_n x & 1-\varepsilon_n y & 1-\varepsilon_n z \\
1-\varepsilon_n x & 1 & 1-\varepsilon_n t & 1-\varepsilon_n u \\
1-\varepsilon_n y & 1-\varepsilon_n t & 1 &  1-\varepsilon_n v \\
1-\varepsilon_n z & 1-\varepsilon_n u & 1-\varepsilon_n v & 1
\end{pmatrix}.
\end{align*}
We call such a sequence a {\it normalized sequence}. 

\begin{lemma}
\label{2016-01-17-lemma}
A  normalized sequence $\{A_n\}$ satisfies
\begin{align}
\label{2026-01-16-RAn}
R(A_n)= \frac{t\,P(x,z,u)\,P(y,z,v)}{16\, x yzuv} +o(1),
\end{align}
where 
\[
P(\alpha,\beta,\gamma):=2(\alpha\beta+\alpha\gamma+\beta\gamma)-(\alpha^2+\beta^2+\gamma^2).
\]
\end{lemma}

\begin{proof}
Since $a_{ij}^{n}=1-\varepsilon_n x_{ij}^n$, with $\{x_{ij}^n\}$ bounded, we have
\[
1-(a_{ij}^n)^2 = 1-(1-\varepsilon_n x_{ij}^n)^2 = 2\varepsilon_n x_{ij}^n+O(\varepsilon_n^2).
\]
Substituting into the denominator of $R(A_n)$ leads to
\[
\prod_{ij\in\{12,13,14,24,34\}}\bigl(1-(a_{ij}^{n})^2\bigr)
=
(2\varepsilon_n)^5\, x_{12}^nx_{13}^nx_{14}^n x_{24}^n x_{34}^n + O(\varepsilon_n^6).
\]
For the $3 \times 3$ minors, using the determinant identity
\[
\left|
\begin{array}{ccc}
1&1-\varepsilon \alpha&1-\varepsilon \beta\\
1-\varepsilon \alpha&1&1-\varepsilon \gamma\\
1-\varepsilon \beta&1-\varepsilon \gamma&1
\end{array}
\right|
=
\varepsilon^2\bigl(2(\alpha\beta+\alpha\gamma+\beta\gamma)-(\alpha^2+\beta^2+\gamma^2)\bigr)+O(\varepsilon^3),
\]
we obtain
\begin{align}
\label{2026-01-17-DDPP}
\Delta_{124}(A_n)=\varepsilon_n^2\,P(x,z,u) +O(\varepsilon_n^3)
\,\,\, \mbox{ and } \,\,\,
\Delta_{134}(A_n)=\varepsilon_n^2\,P(y,z,v) +O(\varepsilon_n^3).
\end{align}
Substituting these expansions into $R(A_n)$ yields
\[
R(A_n)=
\frac{(2\varepsilon_n x_{23}^n+O(\varepsilon_n^2))(\varepsilon_n^2 P(x,z,u)+O(\varepsilon_n^3))(\varepsilon_n^2 P(y,z,v)+O(\varepsilon_n^3))}
{(2\varepsilon_n)^5 x_{12}^nx_{13}^nx_{14}^n x_{24}^n x_{34}^n + O(\varepsilon_n^6)}.
\]
Divide the numerator and the denominator by $\varepsilon_n^5$, regroup the terms and cancel a $2$, to obtain \eqref{2026-01-16-RAn}.
\end{proof}

The expansion of the determinant of $A_n$, in terms of $\varepsilon_n$, does not have a constant term, nor  
$\varepsilon_n$ or $\varepsilon_n^2$ terms. It is
\[
\det A_n=\varepsilon_n^3\,Q(x,y,z,u,v,t)+O(\varepsilon_n^4),
\]
where $Q$ is the cubic polynomial
\[
\begin{aligned}
\frac{1}{2}Q(x,y,z,u,v,t)=\;&
-t^{2}z
-tuv
+tuy
+tuz
+tvx
+tvz
-txy
+txz \\
&+tyz
-tz^{2}
-u^{2}y
+uvx
+uvy
+uxy
-uxz\\
&-uy^{2}
+uyz
-v^{2}x
-vx^{2}
+vxy
+vxz
-vyz.
\end{aligned}
\]
The fact that $\det A_n > 0$ implies that 
\begin{align}
\label{2016-01-17-cond1}
Q(x,y,z,u,v,t)\ge 0.
\end{align}
Together with \eqref{2026-01-17-DDPP}, for the other two $3\times 3$ principal  minors we have
\begin{align*}
\Delta_{123} &= \varepsilon^2\,P(x,y,t) + O(\varepsilon^3), \quad
\Delta_{234} = \varepsilon^2\,P(u,v,t) + O(\varepsilon^3)
\end{align*}
and they together imply 
\begin{align}
\label{2016-01-17-cond2}
P(x,z,u) \ge 0, P(y,z,v) \ge 0, P(x,y,t) \ge 0, P(u,v,t) \ge 0.
\end{align}
Finally from above, we have 
\begin{align}
\label{2016-01-17-cond3}
x,y,z,u,v \in [0,1] \mbox{ and } t \in [0,4].
\end{align}

Observe that if $x = 0$, then the condition $P(x,z,u) \ge 0$ simplifies to $-(z-u)^2 \ge 0$, forcing $z=u$, in which case $P(x,z,u)=0$. Thus the ratio on the right-hand side of \eqref{2026-01-16-RAn} is zero. The conclusion is the same if any of the other variables $y,z,u,v,t$ is zero. Thus, we can assume that 
$$
x,y,z,u,v, t > 0.
$$

Define the set 
$$
\mathcal{FR} := \{(x,y,z,u,v,t) : \eqref{2016-01-17-cond1}, \eqref{2016-01-17-cond2}, \mbox{ and }\eqref{2016-01-17-cond3} \mbox{ hold} \}
$$
and the functions
$$
\mathcal{P}(x,y,z,u,v,t)
=\frac{t}{16z}\,A(x,z,u)\,B(y,z,v),
$$
where
$$
A(x,z,u):=\frac{P(x,z,u)}{xu}
\,\,\, \mbox{ and } \,\,\,
B(y,z,v):=\frac{P(y,z,v)}{yv}.
$$
The importance of Lemma~\ref{2016-01-17-lemma} lies in the realization that
\[
\sup\{ \ \limsup_{n\to\infty}R(A_n) : \{A_n\} \mbox{ is a normalized sequence} \}=  \sup_{\mathcal{FR} } \mathcal{P}(x,y,z,u,v,t).
\]
Thus, it suffices to find the supremum on the right-hand side. 
The next lemma simplifies the situation. 

\begin{lemma}
\label{2026-01-17-lem}
If $P(x,z,u)\ge 0, P(y,z,v)\ge 0$, and $Q(x,y,z,u,v,t)\ge 0$, then 
\[
P(x,y,t)\ge 0 \mbox{ and } P(u,v,t)\ge 0.
\]
\end{lemma}

\begin{proof}
Let $p:=x+y-t$, $q:=x+z-u$, $r:=y+z-v$, 
and consider the matrix
\begin{align*}
N(x,y,z,u,v,t):=
\begin{pmatrix}
2x & p & q\\
p & 2y & r\\
q & r & 2z
\end{pmatrix}.
\end{align*}
A direct calculation shows that $\det N = Q \ge 0$ and 
$$
P(x,z,u) = \det N[\{1,3\}],
P(y,z,v) = \det N[\{2,3\}],
P(x,y,t) = \det N[\{1,2\}].
$$
Write $N$ in block form separating the third coordinate:
\[
N=
\begin{pmatrix}
B & w\\
w^{\mathsf T} & 2z
\end{pmatrix}, 
\mbox{ where }
B:=
\begin{pmatrix}
2x & p\\
p & 2y
\end{pmatrix}
\mbox{ and }
w:=
\begin{pmatrix}
q\\ r
\end{pmatrix}.
\]
Since $z>0$, the Schur complement of the $1 \times 1$ block $2z$ is
\[
S:=B-\frac{1}{2z}ww^{\mathsf T} \mbox{ implying that } \det N=2z \det S.
\]
Moreover, the diagonal entries of $S$ satisfy
\[
S_{11}=2x-\frac{q^2}{2z}=\frac{4xz-q^2}{2z}=\frac{P(x,z,u)}{2z}\ge 0,
\]
\[
S_{22}=2y-\frac{r^2}{2z}=\frac{4yz-r^2}{2z}=\frac{P(y,z,v)}{2z}\ge 0.
\]
Since $\det N \ge 0$, we also have $\det S \ge 0$, and this together with the non-negativity of the diagonal entries imply that 
$S\succeq 0$.  In particular, this shows 
\[
B = S + \frac{1}{2z}ww^{\mathsf T}\succeq 0.
\]
Therefore $P(x,y,t) = 4xy-(x+y-t)^2 = \det B\ge 0$.

The original determinantal ratio is invariant under the  
 swapping of the indexes $1 \leftrightarrow 4$ and $2 \leftrightarrow 3$. This translates to the fact that the current
 situation is invariant under the simultaneous 
 swapping of $x \leftrightarrow u$ and $y\leftrightarrow v$.
Thus, the conclusion
$P(x,y,t)\ge 0$ becomes precisely $P(u,v,t)\ge 0$. Either that or just repeat the above arguments with the two pairs of variables swapped. 
\end{proof}

Two important facts are contained in the proof of Lemma~\ref{2026-01-17-lem}.

\begin{corollary}
The feasible region of our optimization problem is 
$$
\mathcal{FR}:=\left\{ (x,y,z,u,v,t) :
N(x,y,z,u,v,t) \succeq 0 \right\} \cap \{x,y,z,u,v \in [0,1], t \in [0,4]\},
$$
and, by the linearity of $N(x,y,z,u,v,t)$, it is convex.
It is invariant under the simultaneous 
 swapping of $x \leftrightarrow u$ and $y\leftrightarrow v$.
\end{corollary}

As a function of $t$, the equation $Q(x,y,z,u,v,t)=0$ is a quadratic polynomial in $t$ with negative leading coefficient $-2z$ and discriminant $P(x, z, u)P(y, z, v) \ge 0$. Let $t_{\pm}(x,y,z,u,v)$ denote its largest and smallest root. 
Using the connection $Q(x,y,z,u,v,t)=\det N(x,y,z,u,v,t)$,  it is easy to see that 
$$
\{t:\ N(x,y,z,u,v,t)\succeq 0\}=[t_-(x,y,z,u,v),\,t_+(x,y,z,u,v)].
$$
The first part of the next  lemma is standard in semidefinite programming.

\begin{lemma}
\label{2026-01-17-lastL}
Fix values $(x,y,z,u,v)$, such that $(x,y,z,u,v,t) \in \mathcal{FR}$ for some $t \ge 0$. For 
$s \in [0,1]$ let
\begin{align}
\label{2026-01-17-as}
a(s):= (1-s)(x,y,z,u,v) + s(u,v,z,x,y).
\end{align}

(a) The function $f(s):=t_+(a(s))$ is concave on $[0,1]$ with $f(s)=f(1-s)$.

(b) The function $g(s):=A(x,z,u)B(y,z,v)$ increases on $[0,1/2]$
and decreases on $[1/2,1]$.

Thus, both functions achieve their maximum at $s=1/2$.
\end{lemma}

\begin{proof}
(a) By the symmetry and convexity of the feasible region, if $(x,y,z,u,v,t) \in \mathcal{FR}$ for some $t \ge 0$, then 
$(a(s), t) \in \mathcal{FR}$ for all $s \in [0,1]$.

Take any \(s_1,s_2\in[0,1]\) and \(\lambda\in[0,1]\).  
From $N(a(s_i),f(s_i))\succeq 0$, $i=1,2$, and the convexity of the positive semidefinite cone, we get
\begin{align*}
0 &\preceq \lambda N(a(s_1),f(s_1))+(1-\lambda)N(a(s_2),f(s_2)) \\
&= N\bigl(\lambda a(s_1)+(1-\lambda)a(s_2),\lambda f(s_1)+(1-\lambda)f(s_2)\bigr) \\
&= N\bigl(a(\lambda s_1+(1-\lambda)s_2), \lambda f(s_1)+(1-\lambda)f(s_2) \bigr),
\end{align*}
where we used that $N$ and $a$ are affine functions.  This means that the value \(t=\lambda f(s_1)+(1-\lambda)f(s_2)\) is feasible for the parameter \(a(\lambda s_1+(1-\lambda)s_2)\). Since $f(s)$ denotes the largest root, it follows that
\[
f(\lambda s_1+(1-\lambda)s_2)
\ \ge\
\lambda f(s_1)+(1-\lambda)f(s_2).
\]

(b) Since the variables $x, u$ and $y, v$ in $A(x,z,u)B(y,z,v)$ are separated, we show the lemma for $A(x,z,u)$ only. 
Consider only the first and the fourth component of vector $a(s)$:
\begin{align*}
x_s:=(1-s)x + su \mbox{ and } u_s:=(1-s) u + sx
\end{align*}
and note that $x_s+u_s = x+u$ does not depend on $s$. We have
\begin{align*}
A(x_s,z,u_s) = \frac{P(x_s,z,u_s)}{x_su_s}
=\frac{4x_su_s-(x_s+u_s-z)^2}{x_su_s}
= 4 - \frac{(x+u-z)^2}{x_su_s}.
\end{align*}
A direct computation gives
\begin{align*}
x_su_s
=
\bigl((1-s)x+su\bigr)\bigl((1-s)u+sx\bigr)
=
xu+s(1-s)(x-u)^2.
\end{align*}
Thus $x_su_s$ increases on $[0,1/2]$ and decreases on $[1/2,0]$. The lemma follows.
\end{proof}

\begin{theorem}
We have
$$
\sup_{\mathcal{FR}} \mathcal P(x,y,z,u,v,t)  = \frac{27}{16}.
$$
The supremum is achieved at $x=y=z=u=v=1$ and $t=3$.
\end{theorem}

\begin{proof}
Recall that 
$$
\mathcal P(x,y,z,u,v,t) = \frac{t}{16z} A(x,z,u)B(y,z,v).
$$
Let $\mathcal{FR}_p$ be the projection of $\mathcal{FR}$ onto its first five coordinates.
Take any 
$$
(x,y,z,u,v,t) \in \mathcal{FR}.
$$ 
The largest value of $t$, such that $(x,y,z,u,v,t)$ is still feasible, is $t_+(x,y,z,u,v)$.
Thus, we have
\begin{align}
\label{2026-10-17-sup}
\sup_{\mathcal{FR}} \mathcal P(x,y,z,u,v,t) = \sup_{\mathcal{FR}_p} 
\frac{t_+(x,y,z,u,v)}{16z} A(x,z,u)B(y,z,v).
\end{align}
Recall \eqref{2026-01-17-as}. By Lemma~\ref{2026-01-17-lastL}, on the segment $[a(0), a(1)]$,
the functions $t_+(x,y,z,u,v)$ and $A(x,z,u)B(y,z,v)$ achieve their maximum at the point
$$
\frac{a(0)+a(1)}{2} = \Big(\frac{x+u}{2},\frac{y+v}{2},z,\frac{x+u}{2},\frac{y+v}{2}\Big).
$$
Thus, to find the supremum, it is enough to consider only the points $(x,y,z,x,y)$ from $\mathcal{FR}_p$. In that case the supremum is 
\begin{align*}
\sup_{(x,y,z,x,y) \in \mathcal{FR}_p} &
\frac{t_+(x,y,z,x,y)}{16z} A(x,z,x)B(y,z,y) \\
&= \frac{t_+(x,y,z,x,y)}{16z}\frac{z(4x-z)}{x^2} \frac{z(4y-z)}{y^2} \\
&= \frac{2(x+y)-z+\sqrt{(4x-z)(4y-z)}}{32 z} \frac{z(4x-z)}{x^2} \frac{z(4y-z)}{y^2},
\end{align*}
where we calculated the largest root of the quadratic equation $Q(x,y,z,x,y,t)=0$ with respect to $t$. 
Note that $4x-z \ge 0$ and $4y-z \ge 0$ since $P(x,z,x) \ge 0$ and $P(y,z,y) \ge 0$.
Both the numerator and the denominator are positively homogeneous of degree $5$. Thus we can divide them both by $z^5$ and rename the variables. We will use the same names and will keep in mind that now the range of the variables is not 
bounded from above. (In particular, this makes $z=1$.)  After the change of 
variables  $a:=4x-1 \ge 0 \mbox{ and } b:=4y-1  \ge 0$,
the expression simplifies to 
\begin{align}
\label{2026-01-17-alg}
 \frac{4ab(\sqrt{a}+\sqrt{b})^2}{(a+1)^2(b+1)^2} 
= \Big(\frac{2xy(x+y)}{(x^2+1)(y^2+1)}\Big)^2, 
\end{align}
where we changed the variables again to  \(x=\sqrt a \ge 0\) and \(y=\sqrt b \ge 0\). So it suffices to maximize
\[
F(x,y):=\frac{2xy(x+y)}{(x^2+1)(y^2+1)},\qquad x,y \ge 0.
\]
To find the critical points, it is easier to consider
\[
\Phi(x,y)=\ln F(x,y)
=\ln 2+\ln x+\ln y+\ln(x+y)-\ln(x^2+1)-\ln(y^2+1)
\]
with partial derivatives:
\[
\frac{\partial \Phi}{\partial x}=\frac1x+\frac1{x+y}-\frac{2x}{x^2+1}=0
\,\,\, \mbox{ and } \,\,\,
\frac{\partial \Phi}{\partial y}=\frac1y+\frac1{x+y}-\frac{2y}{y^2+1}=0.
\]
These two equations simplify to
\[
x^2y=2x+y \mbox{ and }
y^2x=2y+x. 
\]
The only non-negative solutions to the system are $(0,0)$ and $(\sqrt{3}, \sqrt{3})$.
On the boundary, \(F(x,y)\to 0\) as \(x\to 0^+\) or \(y\to 0^+\).  
Also, if \(x\to\infty\) with \(y\) fixed then
\[
F(x,y)\sim \frac{2x^2y}{x^2(y^2+1)}=\frac{2y}{y^2+1}\le 1,
\]
and similarly if \(y\to\infty\). Therefore the interior critical point gives
the global maximum. (In addition, the Hessian of $F(x,y)$ is degative definite at $(\sqrt{3}, \sqrt{3})$.) Evaluating at \(x=y=\sqrt3\), gives $F(\sqrt{3}, \sqrt{3}) = 3\sqrt{3}/4$. The square of this number is the maximum of \eqref{2026-01-17-alg} and that of \eqref{2026-10-17-sup}.
\end{proof}

\subsection{The supremum of $R$ on the boundary}
\label{subsec2.3}

We need to investigate the ratio $R$ on the boundary of the domain $D$, given that $|a_{ij}|\not=1$ for all $ij = 12$, $13$, $14$, $23$, $24$, $34$. The boundary is charactarized by the condition that $A$ is a positive semidefinite correlation 
matrix $\det A = 0$. We begin with a rational parametrization of the boundary. We chose four unit vectors in $\R^3$ and let $A$ be their Gram matrix. Since the Gram matrix is invariant under rotations in $\R^3$, we can assume that $u_1:=(1,0,0)$ and then rotate again around $u_1$, until $u_2$ lies in the $xy$-plane, that is
$$
u_2:=\frac{1}{d_2}(1-r^2,\ 2r,\ 0), \mbox{ where } d_2 = 1+r^2 \mbox{ for } r \in \R.
$$
Then we use stereographic projection from the plane onto the sphere $S^2$ to define the other two unit vectors
\[
u_3:=\frac{1}{d_3} (1-p^2-q^2,\ 2p,\ 2q), \mbox{ where } d_3 = 1+p^2+q^2 
\mbox{ for } p,q \in \R
\]
and
\[
u_4:=\frac{1}{d_4}(1-s^2-t^2,\ 2s,\ 2t), \mbox{ where } d_4 = 1+s^2+t^2 
\mbox{ for } s,t \in \R.
\]
Letting $a_{ij}:=\langle u_i,  u_j \rangle$, we obtain
\[
a_{12}=\frac{1-r^2}{d_2},\,\,\,\,
a_{13}=\frac{1-p^2-q^2}{d_3},\,\,\,\,
a_{14}=\frac{1-s^2-t^2}{d_4},
\]
\[
a_{23}=\frac{(1-r^2)(1-p^2-q^2)+4rp}{d_2d_3},\,\,\,\,
a_{24}=\frac{(1-r^2)(1-s^2-t^2)+4rs}{d_2d_4},
\]
\[
a_{34}=\frac{(1-p^2-q^2)(1-s^2-t^2)+4ps+4qt}{d_3d_4}.
\]
Since $A$ is a Gram matrix using four linearly dependent unit vectors, $A$ is positive semidefinite correlation matrix with 
$\det A = 0$. This parametrization covers every such matrix, except the cases when $u_2,u_3$ or $u_4$ is equal to $(-1,0,0)$. This situation is covered when the corresponding parameters approach infinity and this situation appears 
naturally in the course of investigations below. 

Plugging $A$ into the ratio gives $R(A) =: N/D$, where $N$ and $D$ factor as
\begin{align*}
N&:= (s^2 + t^2 + 1)^2\cdot((pr+1)^2 + q^2r^2)\cdot((p- r)^2 + q^2)\cdot(pt - qs)^2\cdot t^2, \\
D&:= ((rs+1)^2 + r^2t^2)\cdot(s^2 + t^2)\cdot((ps+qt+1)^2+(pt-qs)^2) \\
& \hspace{0.5cm} \cdot((p-s)^2 + (q - t)^2)\cdot(p^2 + q^2)\cdot((r-s)^2 + t^2),
\end{align*}
where each factor is a sum of squares. Let $N_1$, $N_2$, $N_3$, $N_4$, and $N_5$ be the five factors in the numerator in their respective order. The numerator is zero when
\begin{align*}
r\not=0, \,\, p=-1/r, \,\, q=0, \mbox{ or }  p=r, \,\, q=0, \mbox{ or } pt=qs, \mbox{ or } t=0.
\end{align*}
Denote the different factors in the denominator of $R$ by $D_1,\ldots, D_6$ in their respective order. The denominator is zero when one or more of the following conditions is satisfied
$$
\begin{aligned}
D_1=0 &\iff t=0,\ s=-\frac1r\ \ (r\neq 0),\\
D_2=0 &\iff s=0,\ t=0,\\
D_3=0 &\iff pt=qs,\ ps+qt=-1,\\
D_4=0 &\iff s=p,\ t=q,\\
D_5=0 &\iff p=0,\ q=0,\\
D_6=0 &\iff r=s,\ t=0.
\end{aligned}
$$
In Appendix A, we show that whenever the denominator converges to zero, the possible limits of the  
ratio $R$ are in all but one cases less than or equal to $1$ and in the remaining case it is less than or equal to $27/16$. 
If only the numerator is zero, then $R=0$ and there is nothing to show. Thus, we can assume that we are at a point where neither $N$ nor $D$ is zero. In that case, we look for the critical points of $R$ by taking the logarithm of the ratio and finding all five partial derivatives. We multiply each partial derivative by the common denominator 
$$
CD:=N_1 \cdots N_5 \cdot D_1 \cdots D_6
$$
in order to convert each partial derivative into a polynomial. Then we find the Gr\"obner basis of the five polynomials and $1-yCD=0$. The last condition ensures that the polynomial $CD$ stays non-zero. The Gr\"obner basis with respect to the $lex$ order of the variables $y,p,q,r,s,t$ gives, the details are omitted, the following families of solutions
\begin{align*}
&\{(s,t,r,q,p)\in\mathbb{R}^5: s=0, t\neq \pm 1, (r,q,p)=(0,0,\pm 1)\}, \\
&\{(s,t,r,q,p)\in\mathbb{R}^5: s=0, t=\pm 1, r=0, p^2+q^2=1 \},\\
&\left\{(s,t,r,q,p)\in\mathbb{R}^5: s=0, t=\pm 1, q=0, r=\dfrac{p-1}{p+1} \right\},\\
&\left\{(s,t,r,q,p)\in\mathbb{R}^5: s=0, t=\pm 1, q=0, r=\dfrac{1+p}{1-p} \right\},\\
&\{(s,t,r,q,p)\in\mathbb{R}^5: s=0, t=\pm 1, (p,q,r)=(0,\pm 1,\pm 1) \}, \\
&\{(s,t,r,q,p)\in\mathbb{R}^5: s^2+t^2=1, s\neq 0, (p,q,r)=(0,\pm 1,\pm 1)\}.
\end{align*}
It is straightforward to verify that on all these families, the value of the ratio $R$ is one. This shows that $R$ does not attain its supremum at a critical point. Thus $R$ must approach its supremum when one or more of its variables approach infinity.

Appendix B, goes over all proper subsets of the variables $\{p,q,r,s,t\}$ and shows that when the variables in that subset go to infinity, while the rest stay bounded, then the limit of $R$ is $0$ in $15$ cases, in $13$ cases it is bounded above by $1$, and in $2$ cases it is bounded above by $(3+2\sqrt{2})/4$.  Thus, it only remains to examine the situation when all variables approach infinity. If they approach infinity at different rates, then the situation easily reduces to one of those considered in Appendix B (just consider the subset of variables that approach infinity at the maximal rate).  In order to avoid introducing new letters, we make the following replacements in $R$ and let $\lambda$ approach infinity
$$
p \to \lambda p, \,\, q \to \lambda q,  \,\, r \to \lambda r,  \,\, s \to \lambda s, \,\, t \to \lambda t.
$$
Thus in the limit the same variables $p,q,r,s,t$ appear and we may assume that they are all non-zero numbers, otherwise we are actually dealing with a subset of the variables, a case considered in Appendix B. A straightforward calculation shows 
\begin{align}
\label{2026-01-23-final}
\lim_{p,q,r,s,t \to \infty} R = \frac{(pt-qs)^2t^2}
{(p^2+q^2)(s^2+t^2)((p-s)^2+(q-t)^2)} \cdot \frac{(r-p)^2+q^2}{(r-s)^2+t^2}.
\end{align}

\begin{lemma}
\label{2026-01-24-lem}
The maximum of \eqref{2026-01-23-final} over all non-zero reals $p,q,r,s,t$ is $27/16$.
\end{lemma}

\begin{proof}
Since the variables are independent, we maximize the second fraction with respect to $r$ first. That is, we 
are looking for the largest $m$, such that the following quadratic equation in $r$ has a real solution
$$
(r-p)^2+q^2 = m((r-s)^2+t^2) \mbox{ or } (1 - m)r^2 + 2(ms - p)r + p^2 + q^2 - m(s^2 + t^2) =0.
$$
This happens when its discriminant is nonnegative:
$$
4(ms-p)^2-4(1-m)(p^2+q^2-m(s^2+t^2))\ge0 
$$
and this is a quadratic inequality in $m$:
$$
-t^2m^2 + ((p-s)^2  + q^2 + t^2)m - q^2 \ge 0.
$$
So the largest possible \(m\) is the larger root:
\begin{align*}
m &=\frac{(p-s)^2+q^2+t^2+\sqrt{((p-s)^2+(q-t)^2)((p-s)^2+(q+t)^2)}}{2t^2} \\
&= \frac{1}{4t^2} (\sqrt A+\sqrt B)^2,
\end{align*}
where for short, we defined
\[
A:=(p-s)^2+(q-t)^2 \,\, \mbox{ and } \,\, 
B:=(p-s)^2+(q+t)^2.
\]
Substituting $m$ in \eqref{2026-01-23-final} eliminates the variable $r$ and we are left to maximize
\begin{align}
\label{2026-01-23-expr}
\frac{1}{4} \frac{(pt-qs)^2}{(p^2+q^2)(s^2+t^2)}
\cdot
&\frac{(\sqrt A+\sqrt B)^2}{A} 
= \frac{1}{4} \frac{(pt-qs)^2}{(p^2+q^2)(s^2+t^2)}
\cdot
\left(1+\sqrt \frac{B}{A} \right)^2.
\end{align}

Let $a:=\sqrt{p^2+q^2}$ and  $b:=\sqrt{s^2+t^2}$.
The signed area of the parallelogram determined by the vectors $(p,q)$ and $(s,t)$ is
\[
\left|
\begin{array}{cc}
p & s \\
q & t
\end{array}
\right|
= pt - qs = ab \sin \theta,
\]
where \(\theta\) is the angle between $(p,q)$ and $(s,t)$. By the triangle inequality 
$$
\sqrt B=\|(p,q) - (s, -t) \|\le a+b,
$$
while $A=\|(p,q)- (s,t)\|^2 = a^2+b^2-2ab\cos\theta$.
Letting \(c=\cos\theta\in(-1,1)\) and substituting everything into \eqref{2026-01-23-expr}, we get the upper bound
\[
\frac{1-c^2}{4}\left(1+\frac{a+b}{\sqrt{a^2+b^2-2abc}}\right)^2.
\]
(Note that if $c = \pm 1$, then \eqref{2026-01-23-expr} is zero and there is nothing to show.)
Dividing the numerator and the denominator of the 
 fraction in the parenthesis by $b$, we can assume that $b=1$. So we need to maximize
\[
\frac{1-c^2}{4}\left(1+\frac{a+1}{\sqrt{a^2+1-2ac}}\right)^2 \mbox{ over } a>0,\ -1<c<1.
\]
This expression is bounded above by
$$
\frac{1-c^2}{4}\left(1+\sqrt{\frac{2}{1-c}}\right)^2
$$
as shown by the identity
$$
2(a^2+1-2ac)-(1-c)(a+1)^2
=(1+c)(a-1)^2\ge 0
$$
and equality between the two is attained at $a=1$.
Letting \(y=\sqrt{1-c} \in(0,\sqrt2)\), removes the square root and the new function to maximize is
$$
\frac{y^2(2-y^2)}{4}\left(1+\frac{\sqrt2}{y}\right)^2
=\frac{1}{4}(2-y^2)(y+\sqrt2)^2.
$$
The derivative factors as $(y+\sqrt{2})^2(1/\sqrt{2}-y)$.
Thus the maximum is attained at $y=1/\sqrt{2}$ with maximal value $27/16$.
\end{proof}

\section{Appendix A}

In this section, we investigate the possible limits of the ratio $R$ when the factors in the denominator go to zero one by one, in pairs, in triples, and so on.

\smallskip

\noindent
{Case 1.} Only $D_1=(rs+1)^2 + r^2t^2=0$, that is $t=0,\ s=-1/r$, $r\neq 0$.
Let $t=\e T$ and $s=-\frac1r+\e S$, for some $T^2+S^2 \not=0$, and simplify the ratio by cancelling $\e^2$ and $r^6$ from the numerator and denominator. Then, we obtain 
$$
\lim_{\e \to 0^+} R = \frac{q^2 T^2}{(p^2+q^2)(S^2+T^2)} \le 1
$$
where we used that $D_5 = p^2+q^2 \not=0$.

\smallskip

\noindent
{Case 2.} Only $D_2=s^2+t^2=0$, that is $s=t=0,$.
Let $s=\e S$ and $t=\e T$, with $S^2+T^2 \not= 0$. After factorization, the numerator has a factor $\e^4$, while the denominator has a factor $\e^2$, the rest of the factors do not converge to $0$. Thus, $\lim_{\e \to 0^+} R=0$.

\smallskip

\noindent
{Case 3.} Only $D_3=(ps+qt+1)^2+(pt-qs)^2=0$, that is $pt=qs$ and $ps+qt=-1$.
Since $D_5=p^2+q^2\neq 0$, the conditions imply
\[
(s,t)=-\frac{1}{D_5}(p,q) \mbox{ and we perturb } (s,t)=-\frac{1}{D_5}(p,q) +\e(S,T)
\]
with $S^2+T^2 \not=0$.
After simplifying and factoring, one cancels $\e^2$ and $(p^2 + q^2)^2$  from the numerator and denominator. Then, we obtain 
$$
\lim_{\e \to 0^+} R =\frac{q^2}{p^2 + q^2} \frac{(S q - T p)^2}{(S^2 + T^2)(p^2 + q^2)} \le 1
$$
by the Cauchy-Schwartz inequality.

\smallskip

\noindent
{Case 4.} Only $D_4=(p-s)^2+(q-t)^2=0$, that is $s=p$ and $t=q$.
Let $s=p+\e S$ and $t=q+\e T$ for some $S^2+T^2 \not= 0$. After cancelling $\e^2$, we obtain
\[
\lim_{\e\to 0} R =
\frac{q^2}{p^2+q^2}\cdot
\frac{(q S - p T)^2}{(S^2+T^2)(p^2+q^2)} \le 1,
\]
by the Cauchy-Schwartz inequality and using that $D_5=p^2+q^2\neq 0$.

\smallskip

\noindent
{Case 5.} Only $D_5=p^2+q^2=0$, that is $p=q=0$.
Let $p=\e P$ and $q=\e Q$, for some $P^2+Q^2 \not=0$. After cancelling $\e^2$, we obtain
\[
\lim_{\e\to 0}  R = \frac{(Pt - Qs)^2}{(s^2 + t^2)(P^2 + Q^2)} \cdot
\frac{(s^2 + t^2 + 1)^2r^2t^2}{(s^2 + t^2)((r-s)^2 + t^2)(r^2t^2 + (1 + rs)^2)} \le 1.
\]
The first fraction is bounded above by one by the Cauchy-Schwartz inequality, while the second fraction is bounded above by one since the difference between its denominator and numerator is the perfect square
$$
((r^2-rs-1)(s^2+t^2)+rs)^2.
$$

\smallskip

\noindent
{Case 6.} Only $D_6=(r-s)^2+t^2=0$, that is $s=r$ and $t=0$.
Let $s=r+\e S$ and $t=\e T$, for some $S^2+T^2 \not=0$ and after cancelling $\e^2$ from the numerator and the denominator, we get 
$$
\lim_{\e\to 0} R =
\frac{q^2}{p^2+q^2}\cdot \frac{T^2}{S^2+T^2} \le 1,
$$
where we also used that $D_5 \not=p^2+q^2=0$.

We now look into the cases when more than one multiple of the denominator approach zero. Many of the case are incompatible, so we mention only the compatible ones.

\smallskip

\noindent
{Case 7.} Only $D_1=D_3=0$. The numerator is $O(\e^6)$, while the denominator is $O(\e^4)$, showing that $\lim_{\e\to 0} R = 0$. The case $D_1=D_4=0$ is equivalent.

\smallskip

\noindent
{Case 8.} Only $D_1=D_5=0$, then let $p = \e P$, $q = \e Q$, $s = -1/r + \e S$, $t = \e T$, with $P^2+Q^2 \not= 0$ 
and $S^2+T^2 \not= 0$. After cancelling $\e^4$ we obtain
\begin{align*}
\lim_{\e\to 0} R =
\frac{Q^2}{P^2+Q^2}\cdot \frac{T^2}{S^2+T^2} \le 1.
\end{align*}

\smallskip

\noindent
{Case 9.} Only $D_2=D_4=0$, then let $p = \e P$, $q = \e Q$, $s = \e S$, $t = \e T$ with $P^2+Q^2 \not= 0$ 
and $S^2+T^2 \not= 0$. After cancelling $\e^6$ we obtain
$$
\lim_{\e\to 0} R =
\frac{T^2}{T^2+S^2}\cdot \frac{(PT-QS)^2}{(P^2+Q^2)((P-S)^2  + (Q-T)^2)} \le 1,
$$
where the first fraction is clearly not bigger than one. The second is also not bigger than one, since the difference between its denominator and numerator is the perfect square
$$
(P^2 - PS + Q^2 - QT)^2.
$$
The cases $D_2=D_5=0$ or $D_4=D_5=0$ or $D_2=D_4=D_5=0$ or $D_4=D_5=D_6=0$ are equivalent.

\smallskip

\noindent
{Case 10.} Only $D_2=D_6=0$, then let $r = \e R$,  $s = \e S$, $t = \e T$ with $S^2+T^2 \not= 0$. After cancelling $\e^4$ we obtain
$$
\lim_{\e\to 0} R =
\frac{(Sq - Tp)^2}{(S^2 + T^2)(p^2 + q^2)}\cdot \frac{T^2}{(R-S)^2 + T^2} \le 1.
$$

\smallskip

\noindent
{Case 11.} Only $D_3=D_6=0$, then let $p = -1/s + \e P$, $q=\e Q$, $r = s + \e R$. The numerator is $O(\e^2)$, while the denominator is $O(1)$, showing that $\lim_{\e\to 0} R = 0$.

\smallskip

\noindent
{Case 12.} Only $D_4=D_6=0$, then let $p = s + \e P$, $q=\e Q$, $r = s + \e R$, $t=\e T$. The numerator is $O(\e^6)$, while the denominator is $O(\e^4)$, showing that $\lim_{\e\to 0} R = 0$.

\smallskip

\noindent
{Case 13.} Only $D_5=D_6=0$, then let $p = \e P$, $q=\e Q$, $r = s + \e R$, $t=\e T$ with $P^2+Q^2 \not= 0$ 
and $R^2+T^2 \not= 0$. After cancelling $\e^4$ we obtain
\begin{align*}
\lim_{\e\to 0} R =
\frac{Q^2}{P^2+Q^2}\cdot \frac{T^2}{R^2+T^2} \le 1.
\end{align*}

\smallskip

\noindent
{Case 14.} Only $D_2=D_4=D_6=0$, then let $p = \e P$, $q=\e Q$, $r =  \e R$, $s=\e S$, $t=\e T$ with 
$P^2+Q^2 \not= 0$, $S^2+T^2 \not= 0$ and $R^2+T^2 \not=0$. After cancelling $\e^8$ we obtain
\begin{align*}
\lim_{\e\to 0} R =
\frac{(PT - QS)^2T^2}{(P^2 + Q^2)(S^2 + T^2)((P-S)^2 + (Q - T)^2)} \cdot \frac{(R-P)^2 + Q^2}{(R-S)^2 + T^2}.
\end{align*}
This ratio is identical to \eqref{2026-01-23-final} and Lemma~\ref{2026-01-24-lem} shows that its supremum is $27/16$.
Another way to see it is to note that in this case $a_{ij} \to 1^-$ for all $ij = 12$, $13$, $14$, $23$, $24$, $34$ and this situation was investigated in Subsection~\ref{subsec2.2}, where the derivation did not use that $\det A \not= 0$.
The cases $D_2=D_5=D_6=0$ or $D_2=D_4=D_5=D_6=0$ are the same.

\section{Appendix B}

In this section, we examine the different ways in which the point $(p,q,r,s,t)$ may go to infinity. For each {\it proper} subset 
$S$ of variables, we let the variables in the subset approach infinity, while the rest stay bounded. (The case when all variables approach infinity, is considered in the body of the paper.) We state what the limit of $R$ in each case is. Without loss of generality, we may assume that the variables approach infinity at the same rate. Otherwise, just consider the variables from $S$ that go to infinity at the maximal rate. Say $S=\{p,q,r,s\}$, in order to avoid introducing new letters, we make the following replacements in $R$ and let $\lambda$ approach infinity
$$
p \to \lambda p, \,\, q \to \lambda q,  \,\, r \to \lambda r,  \,\, s \to \lambda s.
$$
Thus in the limit the same variables $p,q,r,s$ appear and we may assume that they are all non-zero numbers, otherwise we are actually dealing with a subset of $S$.

It is immediate to check that for the subsets
\begin{align*}
&\{s\}, \{t\},  \\
& \{p,r\},  \{p,s\},  \{q,r\}, \{q,s\}, \{q,t\}, \{r,s\}, \{s,t\}, \\
& \{p,q,r\}, \{p,q,s\}, \{p,r,s\}, \{q,r,s\},\{q,r,t\}, \\
&\{p,q,r,s\},
\end{align*}
the limit of $R$ is zero. 

\smallskip

\noindent
{Case 1.} For the subset $\{p\}$, we have
$$
\lim_{p \to \infty} R =\frac{t^2}{s^2 + t^2} \cdot \frac{(s^2 + t^2 + 1)^2r^2t^2}{(s^2 + t^2)((r-s)^2  + t^2)((rs+1)^2 + r^2t^2)} \le 1,
$$
since both fractions are bounded by $1$. For the first one that is obvious, while for the second, one needs to notice that the difference between its denominator and numerator is the perfect square
$$
(r^2s^2 + r^2t^2 - rs^3 - rst^2 + rs - s^2 - t^2)^2.
$$

\smallskip

\noindent
{Case 2.} For the subset $\{q\}$, we have
$$
\lim_{q \to \infty} R =\frac{s^2}{s^2 + t^2} \cdot \frac{(s^2 + t^2 + 1)^2r^2t^2}{(s^2 + t^2)((r-s)^2  + t^2)((rs+1)^2 + r^2t^2)} \le 1,
$$
as above.

\smallskip

\noindent
{Case 3.} For the subset $\{r\}$, we have
$$
\lim_{r \to \infty} R =\frac{t^2}{s^2 + t^2} \cdot \frac{(pt - qs)^2(s^2 + t^2 + 1)^2}{((ps+qt+1)^2+(pt-qs)^2)((p-s)^2 + (q-t)^2)(s^2 + t^2)} \le 1,
$$
where the second fraction is bounded by $1$, since the difference between its denominator and numerator is the perfect square
$$
(p^2s^2 + p^2t^2 - ps^3 - pst^2 + q^2s^2 + q^2t^2 - qs^2t - qt^3 + ps + qt - s^2 - t^2)^2.
$$

\smallskip

\noindent
{Case 4.} For the subset $\{p,q\}$, we have
$$
\lim_{p,q \to \infty} R =\frac{(pt-qs)^2}{(p^2+q^2)(s^2+t^2)} \cdot 
\frac{(s^2+t^2+1)^2r^2t^2}{(s^2+t^2)((r-s)^2+t^2)((rs+1)^2+r^2t^2)}\le 1,
$$
where the first fraction is bounded by $1$ by the Cauchy-Schwartz inequality, while the second was encountered in Case 1.

\smallskip

\noindent
{Case 5.} For the subset $\{p,t\}$, we have
$$
\lim_{p,t \to \infty} R =\frac{p^2}{p^2 + t^2} \le 1.
$$

\smallskip

\noindent
{Case 6.} For the subset $\{r,t\}$, we have
$$
\lim_{r,t \to \infty} R =\frac{p^2r^2}{(p^2 + q^2)(r^2 + t^2)} \le 1.
$$

\smallskip

\noindent
{Case 7.} For the subset $\{p,q,t\}$, we have
$$
\lim_{p,q,t \to \infty} R =\frac{p^2}{p^2 + (q-t)^2} \le 1.
$$

\smallskip

\noindent
{Case 8.} For the subset $\{p,r,t\}$, we have
$$
\lim_{p,r,t \to \infty} R =\frac{t^2(p-r)^2}{(r^2 + t^2)(p^2 + t^2)} \le 1,
$$
since the difference between the denominator and the numerator is $(pr + t^2)^2$.

\smallskip

\noindent
{Case 9.} For the subset $\{p,s,t\}$, we have
$$
\lim_{p,s,t \to \infty} R =\frac{t^2}{s^2 + t^2} \cdot \frac{p^2t^2}{(s^2 + t^2)((p-s)^2 + t^2)} \le 1,
$$
since, in the second fraction, the difference between the denominator and the numerator is $(ps - s^2 - t^2)^2$.

\smallskip

\noindent
{Case 10.} For the subset $\{q,s,t\}$, we have
$$
\lim_{q,s,t \to \infty} R = \frac{t^2}{s^2+t^2} \cdot \frac{q^2s^2}{(s^2 + t^2)((q-t)^2 + s^2)}    \le 1,
$$
since, in the second fraction, the difference between the denominator and the numerator is $(qt - s^2 - t^2)^2$.

\smallskip

\noindent
{Case 11.} For the subset $\{r,s,t\}$, we have
$$
\lim_{r,s,t \to \infty} R = \frac{(pt-qs)^2}{(p^2+q^2)(s^2+t^2)} \cdot \frac{r^2t^2}{(s^2+t^2)((r-s)^2+t^2)}    \le 1,
$$
where the first fraction is bounded by $1$ by the Cauchy-Schwartz inequality, while, in the second, the difference between the denominator and the numerator is $(rs - s^2 - t^2)^2$.

\smallskip

\noindent
{Case 12.} For the subset $\{p,q,r,t\}$, we have
\begin{align}
\label{2026-01-22-lim}
\lim_{p,q,r,t \to \infty} R = \frac{p^2t^2}{(p^2 + (q-t)^2)(p^2 + q^2)} 
\cdot \frac{(p-r)^2 + q^2}{r^2 + t^2} \le \frac{3+2\sqrt{2}}{4}.
\end{align}
To show the inequality requires some work. As stated at the beginning of the section, we can assume that all variables are non-zero. Since they are independent of each other, we begin by maximizing the second ratio. In other words, we are looking for the largest $m$ for which the quadratic equation in $r$ has a real solution
\[
(p-r)^2+q^2 = m(r^2+t^2) \mbox{ or } (1-m)r^2-2pr+(p^2+q^2-mt^2) =0.
\]
For this to have a real solution \(r\), its discriminant must be nonnegative:
\[
4p^2-4(1-m)(p^2+q^2-mt^2)\ge 0 \mbox{ or } -t^2m^2+(p^2+q^2+t^2)m- q^2 \ge 0.
\]
Solving this for \(m\) gives the largest possible value
\begin{align*}
m&=\frac{p^2+q^2+t^2+\sqrt{(p^2 + (q-t)^2)(p^2 + (q+t)^2)}}{2t^2} \\
&= \frac{1}{4t^2} \big(\sqrt{p^2 + (q-t)^2}+\sqrt{p^2 + (q+t)^2}\big)^2.
\end{align*}
After replacing the second fraction in \eqref{2026-01-22-lim} by that $m$, and rearranging, we need to maximize
\begin{align}
\label{2026-01-23-max}
\frac{p^2}{4(p^2+q^2)}\left(1+\sqrt{\frac{p^2 + (q+t)^2}{p^2 + (q-t)^2}}\right)^2.
\end{align}
Using homogeneity, we may assume that $p=1$, so we maximize
$$
\frac{1}{4(1+q^2)}\left(1+\sqrt{\frac{1 + (q+t)^2}{1 + (q-t)^2}}\right)^2.
$$
We may assume \(qt > 0\) (otherwise the square root is smaller), and hence we may assume \(q,t > 0\). Setting
\[
k=\frac{2qt}{1+q^2+t^2}\in[0,1),
\]
we need to maximize
\begin{align}
\label{2026-01-22-max}
\frac{1}{4(1+q^2)}\left(1+\sqrt{\frac{1 + k}{1 - k}}\right)^2
\end{align}
By the AM--GM inequality, we have
\[
k=\frac{2qt}{1+q^2+t^2}\le \frac{2qt}{2t\sqrt{1+q^2}}
=\frac{q}{\sqrt{1+q^2}},
\]
with equality at \(t=\sqrt{1+q^2}\). Since the function $\sqrt{(1+k)/(1-k)}$ is increasing, it is maximized at $k=q/\sqrt{1+q^2}$.
At this point, write \(q=\sinh a\) so \(\sqrt{1+q^2}=\cosh a\), where $a > 0$ since $q>0$. Then
\[
\sqrt{\frac{1+k}{1-k}}=\sinh a + \cosh a = e^a.
\]
Hence, expression \eqref{2026-01-22-max} is bounded above by
\begin{align}
\label{2026-01-22-max1}
\frac{(1+e^a)^2}{4\cosh^2 a} =  \left(\frac{u(u+1)}{u^2+1}\right)^2,
\end{align}
where we put \(u=e^a>0\) and used that \(\cosh a=(u+u^{-1})/2\).
We are finally, down to one-variable maximization problem. 
It is now straightforward to check that the function $u(u+1)/(u^2+1)$ is maximized at $u=1+\sqrt{2}$.
Thus, the maximal value of \eqref{2026-01-22-max1} is $(3+2\sqrt{2})/4$.
Retracing the steps, one can see that equality in \eqref{2026-01-22-lim} is 
attained at 
$$
(p,q,r,t)=(p,p,-\sqrt2p,\sqrt2p), \,\, p\ne 0.
$$

\smallskip

\noindent
{Case 13.} For the subset $\{p,q,s,t\}$, we have
$$
\lim_{p,q,s,t \to \infty} R = \frac{t^2}{s^2+t^2} \cdot \frac{(pt - qs)^2}{(s^2 + t^2)((p-s)^2  + (q-t)^2)}  \le 1,
$$
where, in the second fraction, the difference between the denominator and the numerator is $(ps + qt - s^2 - t^2)^2$.

\smallskip

\noindent
{Case 14.} For the subset $\{p,r,s,t\}$, we have
$$
\lim_{p,r,s,t \to \infty} R = \frac{t^2}{s^2+t^2} \cdot 
\frac{t^2(p-r)^2}{((p-s)^2 + t^2)((r-s)^2 + t^2)}  \le 1,
$$
where, in the second fraction, the difference between the denominator and the numerator is $(pr - ps - rs + s^2 + t^2)^2$.

\smallskip

\noindent
{Case 15.} For the subset $\{q,r,s,t\}$, we have
\begin{align}
\label{2026-01-23-ineq}
\lim_{q,r,s,t \to \infty} R = \frac{t^2s^2}{(s^2 + t^2)(t^2+(r-s)^2)} \cdot
\frac{q^2 + r^2}{(q-t)^2 + s^2} \le \frac{3+2\sqrt{2}}{4},
\end{align}
We now show the inequality. Even though the rational function in \eqref{2026-01-23-ineq} resembles the one in \eqref{2026-01-22-lim}, we could not obtain one from the other by relabeling the variables and substitutions. 
As explained, we may assume that the variables are non-zero. We may change signs between the pairs $\{r,s\}$ and 
$\{q,t\}$ to assume that \(s,t>0\). Since the variables are independent, we start by maximizing the second ratio with respect to $q$. In other words, we are looking for the largest $m$, such that the following quadratic equation in $q$ has a real solution
$$
q^2+r^2 = m((q-t)^2+s^2) \mbox{ or } (1-m)q^2 + 2tmq +r^2-m(t^2+s^2) =0.
$$
This equation has a real solution \(q\) when the discriminant is non-negative:
\[
-s^2m^2+(r^2+s^2+t^2)m-r^2\ge 0.
\]
Hence the largest possible value of $m$ is the larger root:
\begin{align*}
m &= \frac{r^2+s^2+t^2+\sqrt{(t^2+(r+s)^2)(t^2+(r-s)^2)}}{2s^2} \\
&= \frac{1}{4s^2}\big(\sqrt{t^2+(r-s)^2} + \sqrt{t^2+(r+s)^2} \big)^2
\end{align*}
Substituting $m$ into \eqref{2026-01-23-ineq} eliminates $q$ and we need to maximize
\[
\frac{t^2}{4(s^2+t^2)} \left(1+\sqrt{\frac{t^2+(r+s)^2}{t^2+(r-s)^2}} \right)^2.
\]
Now, this can be obtained from \eqref{2026-01-23-max} by a variable substitution and maximized analogously. 

{\bf Acknowledgments: } I am very grateful to the Technical University of Vienna and particularly to the department of Variational Analysis, Dynamics, and Operations Research with head Dr. Aris Daniilidis for their warm hospitality during the work on this paper. I am also thankful to Stoyan Apostolov for the good moments we had together while discussing the problem.

\bibliographystyle{amsplain}

\begin{thebibliography}{3}

\bibitem{Barrett:1989}
W.W.~Barrett, C.R.~Johnson and M.~Lundquist: {Determinantal Formulae for Matrix Completions Associated with Chordal Graphs},
{\it Linear Algebra Appl.} {\bf 121} 265--289 (1989).


\bibitem{Braun:2023}
P.~Braun and H.~Sendov: {On the Hadamard-Fischer inequality, the inclusion-exclusion formula, and bipartite graphs}, {\it Linear Algebra Appl.} {\bf 668} 64--92 (2023).

\bibitem{Choi:2016}
D.~Choi: {Determinantal Inequalities of Positive Definite Matrices}, {\it Math. Inequal. Appl.} {\bf 19}(1) 167--172 (2016).

\bibitem{Dong:2022}
S.~Dong, Q.~Wang and L.~Hou: {Determinantal Inequalities for Block Hadamard Product and Khatri-Rao Product of Positive Definite Matrices}, {\it AIMS Math.} {\bf 7}(6) 9648--9655 (2022).

\bibitem{Fallat:2001}
S.H.~Fallat and C.R.~Johnson: {Mutliplicative Principal-Minor Inequalities for Tridiagonal Sign-Symmetric P-Matrices}, {\it Taiwan. J. Math.} {\bf 5}(3) 655--665 (2001).

\bibitem{Fallat:2003}
S.H.~Fallat, M.I. Gekhtman and C.R.~Johnson: {Multiplicative Principal-Minor Inequalities for Totally Nonnegative Matrices}, {\it Adv. Appl. Math.} {\bf 30} 442--470 (2003).


\bibitem{Fiedler:1964}
M.~Fiedler: {Relations between the diagonal elements of two mutually inverse positive definite matrices}, {\it Czechoslovak Math. J.} {\bf 14}(1) 39--51 (1964).

\bibitem{Fu:2017}
X.~Fu, Y.~Liu and S.~Liu: {Extension of Determinantal Inequalities of Positive Definite Matrices}, {J. Math. Inequalities} {\bf 11}(2) 355--359 (2017).

\bibitem{TracyHall:2008}
H.T.~Hall and C.R.~Johnson: {Bounded Ratios of Products of Principal Minors of Positive Definite Matrices}, {\it arXiv} (2008).

\bibitem{Horn:1990}
{R.A.~Horn and C.R.~Johnson}: {\it Matrix {A}nalysis}, Cambridge University Press, 1990.

\bibitem{Jiang:2019}
X.~Jiang, Y.~Zheng and X.~Chen: {Extending a Refinement of Koteljanskii's Inequality}, {\it Linear Algebra Appl.} {\bf 574} 252--261 (2019).


\bibitem{Johnson:1985}
C.R.~Johnson and W.W.~Barrett: {Spanning-tree Extensions of the Hadamard-Fischer Inequalities},
{\it Linear Algebra Appl.} {\bf 66} 177--193 (1985).

\bibitem{Johnson:1993}
C.R.~Johnson and W.W.~Barrett: {Determinantal Inequalities for Positive Definite Matrices}, {\it Discrete Math.} {\bf 119}(1) 97--106 (1993).

\bibitem{Johnson:2014}
C.R.~Johnson and S.K.~Narayan: {The Koteljanskii Inequalities for Mixed Matrices},
{\it Linear Multilinear Algebra} {\bf 62}(12) 1583--1590 (2014).



\bibitem{Yanxi:2013}
L.~Yanxi: {On Improvements of Fischer’s Inequality and Hadamard’s Inequality for $K_0$-Matrices}, {\it J. Inequal. Appl.} {\bf 2013}, 460 (2013).

\end{thebibliography}

\section{Statements \& Declarations}

The author was partially supported by the Natural Sciences and Engineering Research Council (NSERC) of Canada. (Grant number RGPIN-2020-06425.)

The author has no relevant financial or non-financial interests to disclose.

\end{document}